\documentclass[12pt,amssymb]{amsart}

\usepackage{graphicx}  
\usepackage[abs]{overpic}

\usepackage[usenames]{color}
\usepackage{txfonts}

\usepackage[dvipdfm,bookmarks=true,bookmarksnumbered=true,bookmarkstype=toc,colorlinks=true,linkcolor=blue,citecolor=OliveGreen]{hyperref}

\newtheorem{thm}{Theorem}
\newtheorem{lem}[thm]{Lemma}

\theoremstyle{definition}

\theoremstyle{remark}

\newcommand{\refeq}[1]{\textup{(\ref{eq:#1})}}
\newcommand{\refthm}[1]{Theorem \ref{thm:#1}}
\newcommand{\reflem}[1]{Lemma \ref{lem:#1}}

\numberwithin{equation}{section}
\numberwithin{thm}{section}

\newcommand{\const}{C}

\newcounter{const}
\newcommand{\cnst}{\refstepcounter{const}\const_{\theconst}}
\newcommand{\clabel}[1]{\cnst\label{#1}}
\newcommand{\refc}[1]{\const_{\ref{c:#1}}} 

\newcommand{\qtext}{\quad\text}

\newcommand{\sm}{\setminus}

\newcommand{\grad}{\nabla}

\newcommand{\bd}{\partial}         

\newcommand{\ol}[1]{\overline{#1}} 
\newcommand{\til}[1]{\tilde{#1}} 
\newcommand{\str}[1]{{#1}^*}

\newcommand{\R}{{\mathbb R}}
\newcommand{\Z}{{\mathbb Z}}

\newcommand{\inv}{^{-1}}

\newcommand{\normi}[1]{\|#1\|_\infty} 

\makeatletter

\newcommand{\pder}[2]{\if@display\dfrac{\partial#1}{\partial#2}
\else\partial#1/\partial#2\fi}

\newcommand{\union}{\mathop{
\if@display
\bigcup\else\operatorname{\hbox{\small$\bigcup$}}
\fi}}

\makeatother

\newcommand{\di}{n}			    
\newcommand{\Rd}{{\R^\di}}		

\newcommand{\dom}{D}
\newcommand{\bdy}{\bd\dom}

\begin{document}

\title[Remark on Euler flow around a corner]{Remark on single exponential bound of the vorticity gradient for the two-dimensional Euler flow around a corner}

\author{Tsubasa Itoh}
\address{
Department of Mathematics,
Tokyo Institute of Technology,
Oh-okayama Meguro-ku Tokyo 152-8551, Japan
}
\email{tsubasa@math.titech.ac.jp}

\author{Hideyuki Miura}
\address{
Graduate School of Information Science 
and Engineering Mathematical and Computing Sciences, 
Tokyo Institute of Technology,
Oh-okayama Meguro-ku Tokyo 152-8551, Japan
}
\email{miura@is.titech.ac.jp}

\author{Tsuyoshi Yoneda}
\address{
Department of Mathematics,
Tokyo Institute of Technology,
Oh-okayama Meguro-ku Tokyo 152-8551, Japan
}
\email{yoneda@math.titech.ac.jp}



\begin{abstract}
In this paper, the two dimensional Euler  flow under a simple symmetry condition with hyperbolic structure in  a unit square
$\dom=\{(x_1,x_2):0<x_1+x_2<\sqrt{2},0<-x_1+x_2<\sqrt{2}\}$ 
  is considered. It is shown that the Lipschitz estimate of the vorticity  on the boundary is 
\textit{at most} single exponential growth
near the stagnation point.
\end{abstract}

\maketitle





\section{Introduction}

Let $\dom$ be a two-dimensional domain. We 
consider the Euler equation on $\dom$, 
in the vorticity formulation:
\begin{equation}\label{eq:Euler}
\omega_t+(u\cdot\grad)\omega=0,
\quad \omega(x,0)=\omega_0(x).
\end{equation}
Here $\omega$ is the fluid vorticity, and 
$u$ is the velocity 
of the flow determined by the Biot-Savart law. 
We impose a no flow condition for the velocity at the boundary: 
$u \cdot n = 0$ on $\partial D$, 
where $n$ is the unit normal vector on the boundary.
This implies the formula:
\begin{equation}\label{eq:u}
u(x,t)=\grad^{\bot}\int_{\dom}G_{\dom}(x,y)\omega(y,t)dy,
\end{equation}
where $G_{\dom}$ is the Green function for the Dirichlet problem in $\dom$ 
and $\grad^{\bot}=(\partial_{x_2},-\partial_{x_1})$.
Global regular solutions to \eqref{eq:Euler} 
was proved by Wolibner \cite{W} and H\"{o}lder \cite{H}
and there are huge literature on this problem.
We are concerned with the question 
how fast the maximum of the gradient of the vorticity can grow 
as $t \rightarrow \infty$. 
When $\dom$ is a smooth bounded domain, 
the best known upper bound on the growth is double-exponential \cite{Y},
but the question whether such upper bound is sharp 
has been open for a long time. 
Very recently, 
Kiselev and Sverak \cite{KS} answered the question affirmatively
for the case the domain is a ball. 
They gave an example of the solution growing with double exponential rate.
On the other hand, Kiselev and Zlatos \cite{KZ} considered
the 2D Euler flow on some bounded domain with certain cusps.
They showed that the gradient of vorticity blows up  at the cusps 
in finite time. 
These solutions are constructed by imposing certain symmetries on 
the
initial data, which leads 
to a \textit{hyperbolic flow scenario} 
near a stagnation point on the boundary. 
More precisely, 
by the hyperbolic flow scenario,
particles on the boundary (near the stagnation point) head for the stagnation point for all time.
Moreover the relation between this scenario and 
the geometry of the boundary plays a crucial role in 
the double exponential growth or the formation of the singularity.  

Thus it would be an interesting question to ask how the geometry of the 
boundary affects the growth of the solution in the hyperbolic flow scenario. 
In this paper, we consider a unit square
$\dom=\{(x_1,x_2):0<x_1+x_2<\sqrt{2},0<-x_1+x_2<\sqrt{2}\}$,  
and under a simple symmetry condition 
the growth of 
the Lipschitz norm of
the hyperbolic flow on the boundary is shown to be 
\textit{at most} single exponential 
near the stagnation point.

To state our result precisely, 
we rewrite \eqref{eq:Euler} into the whole plane $\mathbb{R}^2$, 
by change of valuables and reflection argument.
By \cite{Taniuchi,TTY} (see also \cite{S}), 
we see that the solution uniquely exists  in
\begin{equation*}
u\in C([0,\infty);L^\infty),\quad 
\omega\in L^\infty_{loc}
([0,\infty):
L^\infty).
\end{equation*}
Note in this case, the initial vorticity may not be continuous. 
However we see that 
$-\nabla (-\Delta)^{-1}\omega= u \in W^{1,1}_{loc}$, 
since the Riesz transform $\nabla (-\Delta)^{-1/2}$ is bounded
from $L^\infty$ to $BMO$. 
Then by the ODE theory of DiPerna-Lions \cite{DL}[Theorem III.2], 
we can define the trajectory $\gamma_X(t)$ corresponding to 
the 2D Euler equation:
\begin{equation*}
\frac{d\gamma_X(t)}{dt}=u(\gamma_X(t),t),\quad \gamma_X(0)=X\in\mathbb{R}^2
\end{equation*}
(c.f. \cite{Z}). The following is the main theorem.

\begin{thm}\label{thm:main}
Let $\omega_0$ be a Lipschitz function and 
odd with respect to the vertical axis, namely, $\omega_0(x_1,x_2)=-\omega_0(-x_1,x_2)$. 
Then there is a universal constant $\const$ such that the following
 statement holds true:
For any $T>0$, there exist $\delta>0$ and the solution of \eqref{eq:Euler} 
such that  
\[
|\omega(x,t)|\le \|\omega_0\|_{\rm{Lip}} |x| e^{\const \normi{\omega_0}t}
\qtext{ for $x\in\bdy\cap B(0,\delta)$ and $t<T$}.
\]
The above estimate exhibits single exponential bound of the Lipschitz norm
at the origin. 
\end{thm}
\noindent
{\bf Remark.}  (1)\ 
We have showed the upper bound of certain hyperbolic flows 
around the stagnation point. It seems likely that our example grows 
exactly with the single exponential rate, but we were not able to
prove it so far.

(2)\ For the case when $\dom$ is the unit square 
$[0,1] \times [0,1]$ with the
periodic boundary condition, Zlatos \cite{Z} constructed 
 examples of the $C^{1,\gamma}$ solutions which exhibit single 
exponential growth, where the flows are also of hyperbolic and have 
odd symmetry in both axes.  
Moreover Hoang-Radosz \cite{HR} considered the same problem and 
they showed that the compression of the fluid induced by the hyperbolic flow 
scenario alone is not sufficient for the double exponential growth of 
the smooth solution (see also \cite{D1,D2}).


\section{Proof of the main theorem}
For $x=(x_1,x_2)\in\Rd$ we let $\til x =(-x_1,x_2)$, 
$\ol{x} =(x_1,-x_2)$, 
$\str{x} =(x_2,x_1)$. 
Let $\dom_+ =\{x\in\dom:x_1>0\}$. 
The Green function $G_{\dom}$ is given explicitly by
\[
G_{\dom}(x,y)=\frac{1}{2\pi}\sum_{n\in\Z^2}
\biggl(\log\frac{|x-2m-y||-x-2m-y|}{|\str x-2m-y||-\str x-2m-y|}\biggr),
\]
where $m=(m_1,m_2)=(\frac{n_1-n_2}{\sqrt{2}},\frac{n_1+n_2}{\sqrt{2}})$.
We can obtain the above  $G_{\dom}$ just using reflection and $\pi/4$-rotation (c.f. \cite{Z}).
Since  $\omega(\cdot,t)\in L^{\infty}(\dom)$ is odd with respect to the vertical axis, 
then we have 
\begin{equation}\label{eq:intp}
\int_{\dom}G_{\dom}(x,y)\omega(y,t)dy
=\frac{1}{2\pi}\sum_{n\in\Z^2}\int_{\dom_+}\biggl(\log
\frac{|x-2m-y||-x-2m-y||\ol{\str x}-2m-y||\til{\str x}-2m-y|}
{|\str x-2m-y||-\str x-2m-y||\til x-2m-y||\ol x-2m-y|}\biggr)\omega(y,t)dy.
\end{equation}
We now give the key lemma.
\begin{lem}\label{lem:uj}
Let $\omega(\cdot,t)\in L^{\infty}(\dom)$ be odd with respect to the vertical axis.
Let $a>1$ and let $\dom_a=\{x\in\dom_+: a x_1 \ge  x_2\}$. 
There exists a constant $\clabel{c:uj}$ depending only on 
$a$ such that

\begin{equation}\label{eq:uj}
\sup_{x\in\dom_a\cap B(0,\frac{1}{2})}\biggl|\frac{u_j(x,t)}{x_j}\biggr|
\le \refc{uj}\normi{\omega(\cdot,t)}
\quad (j=1,2).
\end{equation}
\end{lem}

\begin{proof}
Let us only consider $j=1$ because $j=2$ follows by symmetry.
Let $x\in\dom_a\cap B(0,\frac{1}{2})$.
By \refeq{u} and \refeq{intp}, we have
\[
\begin{split}
u_1(x,t)&=\frac{1}{2\pi}\sum_{n\in\Z^2}
\int_{\dom_+}\Biggl[
\frac{x_2-2m_2-y_2}{|x-2m-y|^2}
-\frac{x_2-2m_2-y_2}{|\til x-2m-y|^2}
-\frac{x_2-2m_1-y_1}{|\str x-2m-y|^2}
+\frac{x_2-2m_1-y_1}{|\ol{\str x}-2m-y|^2} \\
&\qquad\qquad\qquad
-\frac{x_2+2m_2+y_2}{|\ol x-2m-y|^2}
+\frac{x_2-2m_2-y_2}{|-x-2m-y|^2}
+\frac{x_2+2m_1+y_1}{|\til{\str x}-2m-y|^2}
-\frac{x_2+2m_1+y_1}{|-\str x-2m-y|^2}
\Biggr]\omega(y,t)dy \\
&=\frac{2x_1}{\pi}\sum_{n\in\Z^2}
\int_{\dom_+}\Biggl[
\frac{(x_2-2m_2-y_2)(2m_1+y_1)}{|x-2m-y|^2|\til x-2m-y|^2}
-\frac{(x_2-2m_1-y_1)(2m_2+y_2)}{|\str x-2m-y|^2|\ol{\str x}-2m-y|^2} \\
&\qquad\qquad\qquad
-\frac{(x_2+2m_2+y_2)(2m_1+y_1)}{|\ol x-2m-y|^2|-x-2m-y|^2}
+\frac{(x_2+2m_1+y_1)(2m_2+y_2)}{|\til{\str x}-2m-y|^2|-\str x-2m-y|^2}
\Biggr]\omega(y,t)dy.
\end{split}
\]
The right above first two terms are essentially the same as those in (2.4) in \cite{Z}.
Due to the extra symmetry, we have two more terms which bring more cancellations.
More precisely, leading terms with singularity as  in Lemma 2.1 \cite{Z} and Lemma 3.1 \cite{KS} do not appear.
Let
\[
\begin{split}
A_n(x,y)&=\frac{(x_2-2m_2-y_2)(2m_1+y_1)}{|x-2m-y|^2|\til x-2m-y|^2}, 
\qquad B_n(x,y)=\frac{(x_2-2m_1-y_1)(2m_2+y_2)}
{|\str x-2m-y|^2|\ol{\str x}-2m-y|^2}, \\
C_n(x,y)&=\frac{(x_2+2m_2+y_2)(2m_1+y_1)}{|\ol x-2m-y|^2|-x-2m-y|^2},
\quad D_n(x,y)=\frac{(x_2+2m_1+y_1)(2m_2+y_2)}
{|\til{\str x}-2m-y|^2|-\str x-2m-y|^2}.
\end{split}
\]
We need  to divide into two parts:
$n\not=(0,0)$ case and $n=(0,0)$ case. 
If $n\neq(0,0)$, we can control each terms for $|n|^{-3}$ order. Thus, summation of these terms converge. 
More precisely, we have the following calculation for $n\not =(0,0)$:
\[
\begin{split}
&|A_n(x,y)-B_n(x,y)|  \\
\le&\biggl|\frac{-4m_1 m_2}{|x-2m-y|^2|\til x-2m-y|^2}+
\frac{4m_1m_2}{|\str x-2m-y|^2|\ol{\str x}-2m-y|^2}\biggr|+\const |n|^{-3} \\
=&\frac{|4m_1m_2||(|x|^2+|2m+y|^2)(2m_1-2m_2+y_1-y_2)x_2-|x|^2[(2m_1+y_1)^2-(2m_2+y_2)^2]|}{|x-2m-y|^2|\til x-2m-y|^2|\str x-2m-y|^2|\ol{\str x}-2m-y|^2}+\const |n|^{-3} \\
\le&\const |n|^{-3}
\end{split}
\]
and 
\[
\begin{split}
&|-C_n(x,y)+D_n(x,y)|  \\
\le&\biggl|\frac{-4m_1 m_2}{|\ol x-2m-y|^2|-x-2m-y|^2}+
\frac{4m_1m_2}{|\til{\str x}-2m-y|^2|-\str x-2m-y|^2}\biggr|+\const |n|^{-3} \\
=&\frac{|4m_1m_2||(|x|^2+|2m+y|^2)(2m_1+2m_2+y_1+y_2)x_2-|x|^2[(2m_1+y_1)^2-(2m_2+y_2)^2]|}{|\ol x-2m-y|^2|-x-2m-y|^2|\til{\str x}-2m-y|^2|-\str x-2m-y|^2}+\const |n|^{-3} \\
\le&\const |n|^{-3}.
\end{split}
\]
Therefore we obtain
\begin{equation}\label{eq:sum}
\biggl|\sum_{\substack{n\in\Z^2\\ n\neq(0,0)}}
\int_{\dom_+}
\bigl(A_n(x,y)-B_n(x,y)-C_n(x,y)+D_n(x,y)
\bigr)\omega(y,t)dy
\biggr|
\le\const\normi{\omega(\cdot,t)}\sum_{\substack{n\in\Z^2\\ n\neq(0,0)}}|n|^{-3}
\le\const\normi{\omega(\cdot,t)}.
\end{equation}

Next we consider the term $n=(0,0)$.
We show that
\begin{equation}\label{eq:zero}
\Biggl|\int_{\dom_+}
\Biggl[\frac{(x_2-y_2)y_1}{|x-y|^2|\til x-y|^2}
-\frac{(x_2-y_1)y_2}{|\str x-y|^2|\ol{\str x}-y|^2}
-\frac{(x_2+y_2)y_1}{|\ol x-y|^2|-x-y|^2}
+\frac{(x_2+y_1)y_2}{|\til{\str x}-y|^2|-\str x-y|^2}
\Biggr]\omega(y,t)dy
\Biggr|
\le\const\normi{\omega(\cdot,t)}.
\end{equation}
If $y\in\dom_+$ and $|y|\ge2|x|$, 
then
\[
\frac{x_2y_1}{|x-y|^2|\til x-y|^2}
\le\const\frac{|x|}{|y|^3}, 
\quad
\frac{x_2y_2}{|\str x-y|^2|\ol{\str x}-y|^2}\le\const\frac{|x|}{|y|^3},
\]
and
\[
\Biggl|-\frac{y_1y_2}{|x-y|^2|\til x-y|^2}
+\frac{y_1y_2}{|\str x-y|^2|\ol{\str x}-y|^2}
\Biggr|
=\frac{y_1y_2\bigl|4(|x|^2+|y|^2)(y_1-y_2)x_2-4|x|^2(y_1^2-y_2^2)\bigr|}
{|x-y|^2|\til x-y|^2|\str x-y|^2|\ol{\str x}-y|^2} 
\le\const\frac{|x|}{|y|^3}.
\]
We obtain the above estimates by the extra symmetry, and these show that leading terms with singularity as in Lemma 2.1 \cite{Z} and Lemma 3.1 \cite{KS} do not appear. 
The rest calculation are essentially the same as in \cite{KS,Z}.
These inequalities imply that (away from the origin)
\[
\Biggl|\int_{\dom_+\sm B(0,2|x|)}\Biggl[
\frac{(x_2-y_2)y_1}{|x-y|^2|\til x-y|^2}
-\frac{(x_2-y_1)y_2}{|\str x-y|^2|\ol{\str x}-y|^2}
\Biggr]\omega(y,t)dy\Biggr|
\le\const\normi{\omega(\cdot,t)}\int_{2|x|}^{\sqrt{2}}\frac{|x|}{r^2}dr
\le\const\normi{\omega(\cdot,t)}.
\]
In a similar way, we have
\[
\Biggl|\int_{\dom_+\sm B(0,2|x|)}
\Biggl[-\frac{(x_2+y_2)y_1}{|\ol x-y|^2|-x-y|^2}
+\frac{(x_2+y_1)y_2}{|\til{\str x}-y|^2|-\str x-y|^2}
\Biggr]\omega(y,t)dy\Biggr|
\le\const\normi{\omega(\cdot,t)}.
\]
Observe that if $y\in\dom_+\cap B(0,2|x|)$, 
then 
\[
|\ol{\str x}-y|,|\ol x-y|,|-x-y|,|\til{\str x}-y|,|-\str x-y|\ge \const |x|.
\]
Therefore, 
we obtain that (near the origin)
\[
\Biggl|\int_{\dom_+\cap B(0,2|x|)}
\Biggl[-\frac{(x_2+y_2)y_1}{|\ol x-y|^2|-x-y|^2}
+\frac{(x_2+y_1)y_2}{|\til{\str x}-y|^2|-\str x-y|^2}
\Biggr]\omega(y,t)dy\Biggr|
\le\const\normi{\omega(\cdot,t)}|x|^{-2}\int_{\dom_+\cap B(0,2|x|)}dy
\le\const\normi{\omega(\cdot,t)}
\]
and 
\[
\begin{split}
\Biggl|\int_{\dom_+\cap B(0,2|x|)}
\frac{(x_2-y_1)y_2}{|\str x-y|^2|\ol{\str x}-y|^2}
\omega(y,t)dy\Biggr| 
&\le\const\normi{\omega(\cdot,t)}|x|^{-1}\int_{0}^{2|x|}\int_{0}^{2|x|}
\frac{|x_2-y_1|}{|\str x-y|^2}dy_1dy_2 \\
&\le\const\normi{\omega(\cdot,t)}|x|^{-1}\int_{0}^{2|x|}\int_{0}^{2|x|}
\frac{z_1}{z_1^2+z_2^2}dz_1dz_2 \\
&=\const\normi{\omega(\cdot,t)}|x|^{-1}\int_{0}^{2|x|}\arctan 
\frac{2|x|}{z_1}dz_1\le\const\normi{\omega(\cdot,t)}.
\end{split}
\]
Finally we prove that
\[
\Biggl|\int_{\dom_+\cap B(0,2|x|)}
\frac{(x_2-y_2)y_1}{|x-y|^2|\til x-y|^2}
\omega(y,t)dy\Biggr|
\le\const\normi{\omega(\cdot,t)}.
\]
We separate the integral into two region.
Let $\dom_1=\dom_+\cap B(0,2|x|) \cap [0,2x_1]\times[0,1]$ and 
$\dom_2=\dom_+\cap B(0,2|x|) \cap [2x_1,1]\times[0,2x_2]$.
We have
\[
\begin{split}
\Biggl|\int_{\dom_1}
\frac{(x_2-y_2)y_1}{|x-y|^2|\til x-y|^2}
\omega(y,t)dy\Biggr|
&\le\const\normi{\omega(\cdot,t)}\int_{0}^{1}\int_{0}^{2x_1}\frac{|x_2-y_2|x_1}{|x-y|^2|\til x-y|^2}
dy_1dy_2 \\
&\le\const\normi{\omega(\cdot,t)}\int_{0}^{1}\int_{0}^{x_1}\frac{z_2 x_1}{(z_1^2+z_2^2)(x_1^2+z_2^2)}
dz_1dz_2 \\
&=\const\normi{\omega(\cdot,t)}\int_{0}^{1}\frac{x_1}{(x_1^2+z_2^2)}\arctan \frac{x_1}{z_1} dz_2 \\
&\le\const\normi{\omega(\cdot,t)}\arctan \frac{1}{x_1}
\le\const\normi{\omega(\cdot,t)}.
\end{split}
\]
Since $x_2\le a x_1$, we obtain
\[
\begin{split}
\Biggl|\int_{\dom_2}
\frac{(x_2-y_2)y_1}{|x-y|^2|\til x-y|^2}
\omega(y,t)dy\Biggr|
&\le \normi{\omega(\cdot,t)} \int_{\dom_2}
\frac{|x_2-y_2|y_1}{|x-y|^4}dy \\
&\le \const\normi{\omega(\cdot,t)} \int_{0}^{x_2}\int_{x_1}^{1}\frac{z_1 z_2}{(z_1^2+z_2^2)^2}
dz_1dz_2 \\
&\le \const\normi{\omega(\cdot,t)} \int_{0}^{x_2}\frac{z_2}{x_1^2+z_2^2}
dz_2 \\
&\le \const\normi{\omega(\cdot,t)} \log(1+\frac{x_2^2}{x_1^2}) 
\le \const\log(1+a^2)\normi{\omega(\cdot,t)}.
\end{split}
\]
Thus we have \refeq{zero}.
\end{proof}

\begin{proof}[Proof of \refthm{main}]
It is well known that $\normi{\omega(\cdot,t)}=\normi{\omega_0}$ and recall
that if $\omega_0$ is odd with respect to vertical axis, 
then $\omega(x,t)$ is also odd with respect to vertical axis.
Also recall that $\gamma_X(t)=(\gamma_{X,1}(t),\gamma_{X,2}(t))$ is the flow map corresponding to 
the 2D Euler evolution:
\begin{equation}\label{eq:trajectory}
\frac{d\gamma_X(t)}{dt}=u(\gamma_X(t),t),\quad\gamma_X(0)=X.
\end{equation}
Due to the boundary condition on $u$, the trajectories 
which start at the boundary stay on the boundary for all times.
Due to  Lemma \ref{lem:uj}, the following observation holds true: For any $T>0$, there is $\delta>0$ such that
the trajectory starting a point $X\in\bdy\cap B(0,\delta)$ stays in $\gamma_X(t)\in B(0,\frac{1}{2})$ for $t<T$. 
Let $x=\gamma_X(t)$. 
By \reflem{uj} and \refeq{trajectory},
we have
\[
\frac{d\gamma_{X,1}(t)}{dt}
\ge 
=-\refc{uj}\normi{\omega_0}\gamma_{X,1}(t)
\qtext{ for all }t>0.
\]
By Gronwall's lemma we have
$\gamma_{X,1}(t)\ge X_1 e^{-\refc{uj}\normi{\omega_0}t}$, 
so that $\gamma_{x,1}^{\inv}(t)\le x_1 e^{\refc{uj}\normi{\omega_0}t}$.
In a similar way, we obtain $\gamma_{x,2}^{\inv}(t)\le x_2 e^{\refc{uj}\normi{\omega_0}t}$.
Hence we have
\[
|\gamma_x^{\inv}(t)|\le|x| e^{\refc{uj}\normi{\omega_0}t}.
\]
Since $\omega(x,t)=\omega_0(\gamma_x^{\inv}(t))$ by the 2D Euler flow in the Lagrangian form, and $\omega_0$ is Lipschitz,
we obtain
\[
|\omega(x,t)|=|\omega_0(\gamma_x^{\inv}(t))|\le \|\omega_0\|_{\text{Lip}}|\gamma_x^{\inv}(t)|
\le \|\omega_0\|_{\text{Lip}} |x| e^{\refc{uj}\normi{\omega_0}t}.
\]
\end{proof}

\section{Conclusion}
We have seen that under some symmetry conditions of the initial data
for the Euler equations in a unit square, 
the vorticity gradient grows at most single exponential rate along 
the boundary near 
 the stagnation point on the corner.
Taking account of the results \cite{KS,KZ} concerning growing 
solutions in other domains, 
we observe that the angle of the corner of the boundary
is related to the growth of the flows.
It seems likely that the sharper the angle of the corner becomes, 
the slower the growth rate of the solutions under 
the hyperbolic flow scenario gets. 
We will pursue this issue in a future work.
Note that we could not deal with 
the solution on the corner with general angle, 
since the symmetry of the domain plays an important role
in our construction of the solutions. 


\bibliographystyle{amsalpha}
\def\cprime{$'$}
\bibliography{eves}

\end{document}